\documentclass[12pt,leqno]{article}
\usepackage{amsmath,amssymb,amsbsy,amsfonts,amsthm,latexsym,
         amsopn,amstext,amsxtra,euscript,amscd,amsthm,mathdots,hyperref}

\allowdisplaybreaks

\numberwithin{equation}{section}

\newtheorem{lem}{Lemma}

\newtheorem{thm}{Theorem}

\newtheorem{cor}{Corollary}

\newtheorem{rem}{Remark}

\begin{document}

\begin{large}
\centerline{\Large \bf The Vandermonde Determinant of the Divisors}
\centerline{\Large \bf of an Integer}
\end{large}
\vskip 10pt
\begin{large}
\centerline{\sc  Patrick Letendre}
\end{large}
\vskip 10pt
\begin{abstract}
Let $1=d_{1}<d_{2}< \cdots < d_{\tau(n)}=n$ denote the ordered sequence of the positive divisors of an integer $n$. We are interested in estimating the arithmetic function
$$
V(n) := \prod_{1 \le i < j \le \tau(n)}(d_{j}-d_{i}) \quad (n \ge 1).
$$
\end{abstract}


\section{Introduction and notation}

Let $1=d_{1}<d_{2}<\cdots < d_{\tau(n)}=n$ denote the ordered sequence of the positive divisors of $n$. Let $x_{1},\dots,x_{s}$ be arbitrary complex numbers. The Vandermonde matrix is defined by
$$
\mathcal{V}(x_{1},\dots,x_{s}):=\begin{bmatrix}
1 & x_{1} & x_{1}^{2} & \cdots & x_{1}^{s-1} \\
1 & x_{2} & x_{2}^{2} & \cdots & x_{2}^{s-1} \\
\vdots & \vdots & \vdots & \ddots & \vdots \\
1 & x_{s} & x_{s}^{2} & \cdots & x_{s}^{s-1}
\end{bmatrix}
$$
and it is well known that
$$
\det \mathcal{V}(x_{1},\dots,x_{s}) = \prod_{1 \le i < j \le s}(x_{j}-x_{i}).
$$
In this article, we study the function
\begin{eqnarray*}
V(n) & := & \det \mathcal{V}(d_{1},\dots,d_{\tau(n)})\\
& = & \prod_{1 \le i < j \le \tau(n)}(d_{j}-d_{i})
\end{eqnarray*}
for each integer $n \ge 1$. Our goal is to estimate the order of magnitude of $V(n)$. Our estimate will involve the arithmetic function
$$
\Omega_{2}(n):=\sum_{p^{\alpha} \| n}\alpha^{2} \quad (n \ge 1).
$$

\begin{thm}\label{thm:1}
For every integer $n \ge 2$, we have
$$
\frac{\tau(n)^{2}}{4}(\log n)\Bigl(1+O\Bigl(\frac{1}{\log n}+\frac{1}{\sqrt{\Omega_{2}(n)}}\Bigr)\Bigr) \le \log V(n) \le  \frac{3\tau(n)^{2}}{8}(\log n).
$$
\end{thm}

\section{Preliminary results}

\begin{lem}\label{lem:1}
For every integer $n \ge 2$, we have
$$
\max_{X \in \mathbb{R}_{>0}}|\{d \mid n:\ X \le d \le 2X\}| \ll \frac{\tau(n)}{\sqrt{\Omega_{2}(n)}}.
$$
\end{lem}

\begin{proof}
See Lemma 6 from \cite{pl}.
\end{proof}

Let us consider the function
$$
S(n) := \sum_{i=1}^{\tau(n)}(i-1)\log d_{i} \quad (n \ge 1).
$$

In what follows, we will use the well-known identity
\begin{equation}\label{eq:1}
\sum_{d \mid n}\log d =\frac{\tau(n)}{2}\log n.
\end{equation}

\begin{lem}\label{lem:2}
For every integer $n \ge 1$, we have
$$
\frac{\tau(n)^{2}}{4}(\log n) \le S(n) \le  \frac{3\tau(n)^{2}}{8}(\log n).
$$
\end{lem}

\begin{proof}
We begin with the lower bound. Define the auxiliary function
$$
S^{*}(n):=\sum_{i=1}^{\tau(n)}(\tau(n)-i)\log d_{i}.
$$
We then write
\begin{eqnarray*}
2S(n) & \ge & S(n)+S^{*}(n)+(\tau(n)-1)\log n\\
& = & (\tau(n)-1)\sum_{i=1}^{\tau(n)}\log d_{i}+(\tau(n)-1)\log n\\
& = & \frac{(\tau(n)-1)(\tau(n)+2)}{2}\log n\\
& \ge & \frac{\tau(n)^{2}}{2}(\log n)
\end{eqnarray*}
and the desired lower bound follows.

We now turn to the upper bound. Let $1_{\square}(n)$ denote the characteristic function of perfect squares. Using the symmetry relation $\log(d)+\log(n/d)=\log(n)$, we pair the divisors accordingly and obtain
\begin{eqnarray}\nonumber
S(n) & = & \Bigl(\frac{\lfloor\frac{\tau(n)}{2}\rfloor(\lfloor\frac{\tau(n)}{2}\rfloor-1)}{2}+\frac{1_{\square}(n)}{2}\Bigr)\log n+\sum_{i=1}^{\lfloor\frac{\tau(n)}{2}\rfloor}(\tau(n)+1-2i)\log d_{\tau(n)+1-i}\\ \label{lem:2:1}
& \le & \Bigl(\frac{\lfloor\frac{\tau(n)}{2}\rfloor(\lfloor\frac{\tau(n)}{2}\rfloor-1)}{2}+\frac{1_{\square}(n)}{2}\Bigr)\log n+(\log n)\sum_{i=1}^{\lfloor\frac{\tau(n)}{2}\rfloor}(\tau(n)+1-2i)\\ \nonumber
& \le & \frac{3\tau(n)^{2}}{8}\log n.
\end{eqnarray}
\end{proof}

Lemma \ref{lem:2} plays a central role in the proof of Theorem \ref{thm:1}. To confirm its optimality, we prove a more precise result in Corollary \ref{cor:1} below.

\begin{lem}\label{lem:3}
For every integer $n \ge 1$, we have
$$
\sum_{d \mid n}\Bigl(\log d - \frac{\log n}{2}\Bigr)^{2}=\tau(n)\sum_{p^{\alpha} \| n}(\log p^{\alpha})^{2}\frac{(\alpha+2)}{12\alpha}.
$$
\end{lem}

\begin{proof}
Expanding the square, we obtain
\begin{equation}\label{lem:3:1}
\sum_{d \mid n}\Bigl(\log d - \frac{\log n}{2}\Bigr)^{2}=\sum_{d \mid n}(\log d)^{2}-\tau(n)\frac{(\log n)^{2}}{4}
\end{equation}
using the identity \eqref{eq:1}.

Now, let $\Lambda(\cdot)$ denote the von Mangoldt function. Using the identity $\log m=\sum_{e \mid m}\Lambda(e)$, we write
\begin{eqnarray*}
\sum_{d \mid n}(\log d)^{2} & = & \sum_{d \mid n}\Bigl(\sum_{e \mid d}\Lambda(e)\Bigr)^{2}\\
& = & \sum_{e_{1},e_{2} \mid n}\Lambda(e_{1})\Lambda(e_{2})\sum_{\substack{d \mid n \\ [e_{1},e_{2}] \mid d}}1\\
& = & \sum_{e_{1},e_{2} \mid n}\Lambda(e_{1})\Lambda(e_{2})\tau\Bigl(\frac{n}{[e_{1},e_{2}]}\Bigr)\\
& = & \tau(n)\sum_{p^{\alpha} \| n}\frac{(\log p)^{2}}{\alpha+1}\sum_{a_{1},a_{2}=1}^{\alpha}(1+\alpha-\max(a_{1},a_{2}))\\
&  & +\tau(n)\sum_{\substack{p_{1}^{\alpha},p_{2}^{\beta} \| n \\ p_{1} \neq p_{2}}}\frac{(\log p_{1})(\log p_{2})}{(\alpha+1)(\beta+1)}\sum_{\substack{1 \le a \le \alpha \\ 1 \le b \le \beta}}(1+\alpha-a)(1+\beta-b)\\
& = & \tau(n)\sum_{p^{\alpha} \| n}(\log p^{\alpha})^{2}\frac{(\alpha+2)}{12\alpha}+\tau(n)\Bigl(\sum_{p^{\alpha} \| n}\frac{(\log p^{\alpha})}{2}\Bigr)^{2}.
\end{eqnarray*}
Inserting this into \eqref{lem:3:1} completes the proof.
\end{proof}

While Lemma \ref{lem:3} is enough for our purposes, the interested reader may consult \cite{sd:gt} for a much more detailed study of the distribution of the logarithms of the divisors of smooth integers.

\begin{cor}\label{cor:1}
The set of limit points of the function $\frac{S(n)}{\tau(n)^{2}\log n}$ is the interval $[1/4,3/8]$.
\end{cor}

\begin{proof}
In view of Lemma \ref{lem:2}, it suffices to show that every real number in $[1/4,3/8]$ is a limit point. We begin with a preliminary observation. For each $p^{\alpha} \parallel n$, let $\theta_{p^{\alpha}} = \theta_{p^{\alpha}}(n) := \frac{\log p^{\alpha}}{\log n}$ ($n \ge 2$). For each $\delta \in (0,1/2]$, define the quantity
$$
J_{\delta}(n) := |\{d \mid n:\ |\log d-\frac{\log n}{2}| \ge \delta\log n\}|.
$$
We have
\begin{eqnarray}\nonumber
J_{\delta}(n) & \le & \frac{1}{(\delta\log n)^{2}}\sum_{d \mid n}\Bigl(\log d - \frac{\log n}{2}\Bigr)^{2}\\ \nonumber
& \le & \frac{\tau(n)}{(\delta\log n)^{2}}\sum_{p^{\alpha} \| n}(\log p^{\alpha})^{2}\frac{(\alpha+2)}{12\alpha}\\ \nonumber
& \le & \frac{\tau(n)}{4\delta^{2}}\sum_{p^{\alpha} \| n}\theta_{p^{\alpha}}^{2}\\ \label{cor:1:1}
& \le & \tau(n)\frac{\max_{p^{\alpha} \| n}\theta_{p^{\alpha}}}{4\delta^{2}}.
\end{eqnarray}
Fix $\epsilon > 0$ and an integer $n$ large enough such that $\max_{p^{\alpha} \| n}\theta_{p^{\alpha}} \le 4\epsilon^{3}$. From \eqref{cor:1:1}, we have $J_{\epsilon}(n) \le \epsilon \tau(n)$, which means
\begin{equation}\label{cor:1:2}
|\{d \mid n:\ |\log d-\frac{\log n}{2}| < \epsilon\log n\}| \ge (1-\epsilon)\tau(n).
\end{equation}
Choose $\kappa > 2\epsilon$. There exists a prime $p \in [n^{\kappa},n^{\kappa+\epsilon}]$ that does not divide $n$. Define $N:=pn$. We now partition the divisors of $N$ into three subsets. The set $\mathcal{D}_{1}$ contains the divisors of $n$ satisfying $|\log d-\frac{\log(N/p)}{2}| < \epsilon\log(N/p)$. From \eqref{cor:1:2}, $\mathcal{D}_{1}$ has at least $\frac{1-\epsilon}{2}\tau(N)$ elements. The set $\mathcal{D}_{2}$ consists of the divisors of $N$ of the form $pd$ for $d \in \mathcal{D}_{1}$. In particular, $|\mathcal{D}_{2}|=|\mathcal{D}_{1}|$, and each $d \in \mathcal{D}_{2}$ satisfies the inequality $|\log d-\frac{\log(Np)}{2}| < \epsilon\log(N/p)$. Finally, the set $\mathcal{D}_{3}$ consists of the remaining divisors of $N$, and therefore contains at most $\epsilon\tau(N)$ elements. We thus have
\begin{eqnarray*}
\sum_{\substack{d_{i} \mid N \\ d_{i} \in \mathcal{D}_{1}}}(i-1)\log d_{i} & = & \sum_{\substack{d_{i} \mid N \\ d_{i} \in \mathcal{D}_{1}}}(i-1)\frac{\log(N/p)}{2} +O\Bigl(\epsilon(\log(N/p))\sum_{i \le \tau(N)/2}(i-1)\Bigr)\\
& = & \sum_{i \le \tau(N)/2}(i-1)\frac{\log(N/p)}{2} +O\bigl(\epsilon(\log N)\tau(N)^{2}\bigr)\\
& = &  \frac{\tau(N)^{2}}{8}\frac{\log(N/p)}{2} +O\bigl(\epsilon(\log N)\tau(N)^{2}\bigr).
\end{eqnarray*}
We have used the fact that $d \in \mathcal{D}_{1}$ implies $d < N^{1/2}$ from the assumption $\kappa > 2\epsilon$ and the bound $|\log d-\frac{\log(N/p)}{2}| < \epsilon\log(N/p)$. Also, at the third line, we have use the inequality $\tau(N) \ge 2^{\frac{1}{4\epsilon^{3}}}$ to simplify the expression. Similarly, we have
$$
\sum_{\substack{d_{i} \mid N \\ d_{i} \in \mathcal{D}_{2}}}(i-1)\log d_{i}=\frac{3\tau(N)^{2}}{8}\frac{\log(Np)}{2}+O\bigl(\epsilon(\log N)\tau(N)^{2}\bigr).
$$
Finally,
$$
\sum_{\substack{d_{i} \mid N \\ d_{i} \in \mathcal{D}_{3}}}(i-1)\log d_{i} \le \epsilon \tau(N)^{2}\log N.
$$
It follows that
\begin{eqnarray*}
S(N) & = & \frac{\tau(N)^{2}}{8}\frac{\log(N/p)}{2}+\frac{3\tau(N)^{2}}{8}\frac{\log(Np)}{2}+O\bigl(\epsilon(\log N)\tau(N)^{2}\bigr)\\
& = & \frac{\tau(N)^{2}}{4}(\log N)+\frac{\tau(N)^{2}}{8}(\log p)+O\bigl(\epsilon(\log N)\tau(N)^{2}\bigr)\\
& = & \frac{\tau(N)^{2}}{4}(\log N)+\frac{\tau(N)^{2}}{8}\Bigl(\frac{\kappa}{1+\kappa}\log N\Bigr)+O\bigl(\epsilon(\log N)\tau(N)^{2}\bigr)\\
& = & \frac{2+3\kappa}{8+8\kappa}\tau(N)^{2}(\log N)+O\bigl(\epsilon(\log N)\tau(N)^{2}\bigr).
\end{eqnarray*}
The result follows by letting $\epsilon \rightarrow 0$ for a fixed $\kappa$ in this construction.
\end{proof}

\begin{rem}
One can show that when $\tau(n)$ is sufficiently large, the only way for the function $\frac{S(n)}{\tau(n)^2 \log n}$ to be close to $\frac{3}{8}$ is that $n$ has a dominant prime factor $p$ such that $p \| n$. Indeed, in that case, this forces the logarithms of exactly half of the divisors to be close to $\log n$. Otherwise, there are two possibilities: either 
$$
\max_{p^{\alpha} \| n} \theta_{p^\alpha} \le 1 - t
$$ 
for some $0 < t \le \frac{1}{3}$, or 
$$
\max_{p^{\alpha} \| n} \theta_{p^\alpha} > 1 - t
$$
but with the maximum attained at some $p^{\alpha} \| n$ with $\alpha \ge 2$. In both cases, by an argument similar to that leading to \eqref{cor:1:1}, one finds a positive density of divisors of $n$ whose logarithms lie significantly close to $\frac{\log n}{2}$. Therefore, by the symmetry around $\frac{\log n}{2}$, that is, the fact that $|\log(d)-\frac{\log n}{2}| = |\log(n/d)-\frac{\log n}{2}|$, we obtain a positive density of divisors $d$ such that
$$
\log d \in \Bigl[ \frac{\log n}{2}, (1 - c)\log n \Bigr]
$$
for some constant $c > 0$ depending on $t$. Thus, using the first line of \eqref{lem:2:1}, we deduce that $\frac{S(n)}{\tau(n)^2 \log n}$ is significantly smaller than $\frac{3}{8}$.
\end{rem}

\section{Proof of Theorem \ref{thm:1}}

Fix an integer $2 \le j \le \tau(n)$. Using Lemma \ref{lem:1}, we write
\begin{eqnarray*}
\sum_{1 \le i < j}|\log (d_{j}-d_{i})-\log(d_{j})| & = &\sum_{\substack{1 \le i < j \\ d_{i}>d_{j}/2}}|\log\Bigl(1-\frac{d_{i}}{d_{j}}\Bigr)|+\sum_{\substack{1 \le i < j \\ d_{i}\le d_{j}/2}}|\log\Bigl(1-\frac{d_{i}}{d_{j}}\Bigr)|\\
& \ll & \frac{\tau(n)\log n}{\sqrt{\Omega_{2}(n)}}+\sum_{\substack{1 \le i < j \\ d_{i}\le d_{j}/2}}1.
\end{eqnarray*}
It follows that
$$
|\log(V(n))-S(n)| \ll \frac{\tau(n)^{2}\log n}{\sqrt{\Omega_{2}(n)}}+\tau(n)^{2}
$$
and the lower bound then follows from Lemma \ref{lem:2}. The simple upper bound follows from the observation that $\log(V(n)) \le S(n)$, which completes the proof.

{\sc Département de mathématiques et de statistique, Université Laval, Pavillon Alexandre-Vachon, 1045 Avenue de la Médecine, Québec, QC G1V 0A6} \\
{\it E-mail address:} {\tt Patrick.Letendre.1@ulaval.ca}

\end{document}